\documentclass[conference]{IEEEtran}
\IEEEoverridecommandlockouts
\usepackage{cite}
\usepackage{amsmath,amssymb,amsfonts}
\usepackage{algorithmic}
\usepackage{graphicx}
\usepackage{textcomp}
\usepackage{xcolor}
\usepackage{MnSymbol}

\usepackage{mathsemantics}
\usepackage[commentmarkup = footnote]{changes}
\usepackage{todonotes}
\usepackage{url}
\usetikzlibrary{decorations.pathmorphing,arrows.meta}

\usepackage{changepage} 

\makeatletter
\def\widebreve{\mathpalette\wide@breve}
\def\wide@breve#1#2{\sbox\z@{$#1#2$}%
     \mathop{\vbox{\m@th\ialign{##\crcr
\kern0.08em\brevefill#1{0.8\wd\z@}\crcr\noalign{\nointerlineskip}%
                    $\hss#1#2\hss$\crcr}}}\limits}
\def\brevefill#1#2{$\m@th\sbox\tw@{$#1($}%
  \hss\resizebox{#2}{\wd\tw@}{\rotatebox[origin=c]{90}{\upshape(}}\hss$}
\makeatletter

\newcommand{\R}{\mathbb{R}}
\renewcommand{\so}{\mathfrak{so}}

\newcommand{\J}{\mathbb{J}}

\newcommand{\SO}{\normalfont{\text{SO}}}
\newcommand{\SE}{\normalfont{\text{SE}}}

\newtheorem{theorem}{Theorem}[section]

\newenvironment{proof}{\paragraph{Proof}}{\hfill$\square$}
\newtheorem{corollary}[theorem]{Corollary}

\newtheorem{remark}[theorem]{Remark}

\begin{document}

\title{A Generalized Sasaki Metric on the Second-Order Tangent Bundle
\thanks{M. Camarinha was supported by the Centre for Mathematics of the University of Coimbra (CMUC, https://doi.org/10.54499/UID/00324/2025) under the Portuguese Foundation for Science and Technology (FCT), Grants UID/00324/2025 and UID/PRR/00324/2025. J. Goodman was supported by the Marie Skłodowska-Curie grant agreement No. 101206748 (GNACS).}
}

\author{
\IEEEauthorblockN{Margarida Camarinha}
\IEEEauthorblockA{\textit{CMUC, Department of Mathematics} \\
\textit{University of Coimbra}\\
Coimbra, Portugal \\
mmlsc@mat.uc.pt}
\and
\IEEEauthorblockN{Jacob Goodman}
\IEEEauthorblockA{\textit{Department of Mathematical Sciences} \\
\textit{Norwegian University of Science and Technology}\\
Trondheim, Norway \\
jacob.goodman@ntnu.no}
}
\date{}

\maketitle

\begin{abstract}
This paper constructs a connection map on the second-order
tangent bundle induced by a linear connection on the base manifold and
uses it to define a generalized Sasaki metric. The associated geodesic
equations are derived, and jet-constrained variational problems are
shown to yield Riemannian quintics in tension. The construction is then
specialized to rigid body attitude dynamics with first-order actuator
dynamics, producing an intrinsic higher-order trajectory model on the
rotation group. Numerical simulations compare quintics in tension with
Riemannian cubics as nominal trajectories and show modest reductions in
actuator-relevant cost with comparable tracking performance.
\end{abstract}


\section{Introduction}
Classical mechanical control systems are typically second order, arising from Newton's law or, more generally, from the Euler--Lagrange equations of a mechanical Lagrangian. On a smooth manifold \(M\), however, acceleration is not a coordinate-free object unless one specifies additional geometric structure, usually a linear connection \(\nabla\) and, in particular, the Levi--Civita connection of a Riemannian metric \(g\). An equivalent viewpoint is to rewrite the dynamics as a first-order control-affine system on the tangent bundle \(TM\). This perspective is especially prominent for connection control-affine systems, whose geometry is naturally encoded on \(TM\) rather than on \(M\) alone \cite{Lewis}.

Once the dynamics are translated to \(TM\), one is led to ask what geometric structure should be placed on this enlarged state space. In many problems this choice is made implicitly. For instance, quadratic running costs in optimal control, distance-based configuration errors in Lyapunov analysis, and sensitivity estimates for perturbed trajectories all require a notion of metric on the state space. When \(M\) is Riemannian, the classical Sasaki metric provides a canonical prolongation of \(g\) to \(TM\), compatible with the horizontal--vertical splitting induced by \(\nabla\) (see \cite{Sasaki} for a modern treatment). In this sense, the Sasaki metric supplies a natural ambient geometry for second-order systems viewed as first-order systems on \(TM\).

In many applications, however, the second-order description is itself an idealization. Practical actuation is rarely ``direct'': the commanded input acts on internal actuator variables, and the corresponding generalized force is generated through additional actuator dynamics, which cause the effective dynamics of the total system to be higher-order. In such a case, the effective control system can be modeled naturally as a control-affine system on some higher-order tangent bundle, where again we are left with a need for some canonical prolongation of the Riemannian structure on $M$ analogous to the Sasaki metric on $TM$. Many such structures have been proposed in the literature, however a substantial part centers on nonlinear connections induced by semisprays or higher-order Lagrangians \cite{Leon, Bucataru}. Our approach is different in spirit. 

In this work, we begin with a linear connection $\nabla$ on the base manifold \(M\) and construct associated connection maps \(K_1,K_2\) in terms of iterated covariant derivatives on $M$, which themselves induce a canonical nonlinear connection on $T^{(2)}M$ \cite{Suri}. In this way, our construction extends the classical Dombrowski--Sasaki picture on \(TM\) to \(T^{(2)}M\) \cite{Sasaki58}, \cite{Dombrowski}: the connection maps determine a natural splitting of \(TT^{(2)}M\), and the generalized Sasaki metric is obtained by declaring the corresponding direct-sum decomposition be orthogonal. This viewpoint is particularly convenient for computation and for control systems written intrinsically in terms of covariant derivatives, since it does not require a prior choice of semispray or higher-order Lagrangian. With this choice, we derive the corresponding geodesic equations and study several of their structural properties, including the relation between geodesics on \(T^{(2)}M\) and jet lifts of curves on \(M\). We then specialize the construction to \(M=\SO(3)\), where the geometry admits an explicit left-trivialized description, and use this to formulate a control-affine rigid body model on \(T^{(2)}\SO(3)\) with first-order actuator dynamics, together with a natural higher-order optimal control problem induced by the generalized Sasaki metric.

\section{Connection Maps}

\subsection{The Geometry of the Tangent Bundle}

Let \(M\) be a finite-dimensional smooth manifold and let \(\tau_{TM}:TM\to M\) denote the tangent bundle projection. The tangent bundle \(TM\) carries a canonical vertical subbundle
\[
V:=\ker(\tau_{TM*})\subset TTM.
\]
A choice of complementary subbundle \(H\subset TTM\) such that \(TTM=H\oplus V\) is called a \emph{nonlinear connection} (or Ehresmann connection). Equivalently, one may work with a \emph{connection map} \(K:TTM\to TM\), i.e. a map such that \(K|_{(x,v)}:T_{(x,v)}TM\to T_xM\) is linear and
\[
K\circ J=\tau_{TM*},
\]
where \(J:TTM\to TTM\) is the almost tangent structure, defined by
\[
J(X)=(\tau_{TM*}X)^v,\qquad X\in TTM,
\]
with \((\cdot)^v\) denoting the vertical lift. The associated horizontal bundle is then \(H=\ker(K)\), and every nonlinear connection arises uniquely in this way; see \cite{Suri}.

Given a linear connection \(\nabla\) on \(M\), Dombrowski constructed a canonical connection map \(K\) on \(TM\) \cite{Dombrowski}. If \(X\in T_{(x,v)}TM\) and \(\gamma(s,t)\) is adapted to \(X\), in the sense that
\[
\gamma(0,0)=x,\quad \partial_t\gamma(0,0)=v,\quad
\partial_s\big|_{s=0}(\gamma(s,0),\partial_t\gamma(s,0))=X,
\]
then
\begin{equation}\label{eq: connection_map_TM}
K(X)=\nabla_{\partial_t\gamma}\partial_s\gamma\big|_{(0,0)}.
\end{equation}
In what follows, \(\nabla\) will always denote the Levi--Civita connection of the given Riemannian metric \(g\) on \(M\).

The corresponding Sasaki metric on \(TM\) is obtained by requiring the splitting \(TTM=\ker(K)\oplus V\) to be orthogonal. Explicitly,
\begin{equation}\label{eq: Sasaki_metric}
   \llangle X,Y\rrangle =
\langle \tau_{TM*}X,\tau_{TM*}Y\rangle
+
\langle K(X),K(Y)\rangle, 
\end{equation}

for all $X,Y \in T_{(x,v)}TM$. 
If \(\Gamma(t)\) is a curve on \(TM\), write
\[
q(t):=\tau_{TM}(\Gamma(t)),\qquad \Gamma(t)=(q(t),v(t)),
\]
where \(v(t)\) is a vector field along \(q\). The connection map then induces the bundle isomorphism
\[
\Phi:TTM\to TM\oplus TM,\qquad \Phi(X):=(\tau_{TM*}X,K(X)),
\]
from which it can be seen that
\[
\Phi(\dot\Gamma(t))=(\dot q(t),\nabla_{\dot q}v(t)).
\]
Hence the Sasaki energy can be written as
\[
E^{(1)}[\Gamma]
=
\frac12\int_0^T \llangle \dot\Gamma,\dot\Gamma\rrangle\,dt
=
\frac12\int_0^T \left(\|\dot q\|^2+\|\nabla_{\dot q}v\|^2\right)\,dt.
\]
 
The critical points of \(E^{(1)}\) satisfy
\begin{equation}\label{eq: sasaki_geodesic_TM}
\nabla_{\dot q}\dot q+\mathcal R(v,\nabla_{\dot q}v)\dot q=0,
\qquad
\nabla_{\dot q}^2v=0,
\end{equation}
where \(\mathcal R\) denotes the curvature endomorphism of \(\nabla\). In particular, if \(K(\dot\Gamma)=0\), equivalently \(\nabla_{\dot q}v=0\), then \eqref{eq: sasaki_geodesic_TM} reduces to
\[
\nabla_{\dot q}\dot q=0,\qquad \nabla_{\dot q}v=0,
\]
so horizontal geodesics of the Sasaki metric project to geodesics of \(M\). As a consequence, the first jet lift \(j^1q\) is a geodesic of the Sasaki metric on \(TM\) if and only if \(q\) is a geodesic on \(M\). 

For \(M=\SO(3)\), the tangent bundle trivializes via left translations as \(T\SO(3)\cong \SO(3)\times\R^3\). If \(\Gamma(t)=(R(t),V(t))\), we write
\[
\hat\eta=R^TV,\qquad \hat\Omega=R^T\dot R,\qquad \hat\Lambda=R^T\nabla_{\dot R}V,
\]
where \(\hat{\cdot}:\R^3\to\so(3)\) is the hat isomorphism given by
\[\begin{bmatrix}
    x \\ y\\ z
\end{bmatrix}^\wedge = \begin{bmatrix} 0 & -z & y \\ z & 0 & -x \\ -y & x & 0 \end{bmatrix}. \]
In these coordinates, the Sasaki geodesic equations become
\begin{align*}
\dot R &= R\hat\Omega, &&\dot\Omega = \frac14\,\Omega\times(\Lambda\times\eta),\\
\dot\eta &= \Lambda+\frac12\,\eta\times\Omega, &&\dot\Lambda = \frac12\,\Lambda\times\Omega,\\
\end{align*}
using the standard identities for the bi-invariant metric on \(\SO(3)\),
\[
\nabla_AB=\frac12\,A\times B,
\qquad
\mathcal R(A,B)C=-\frac14\,(A\times B)\times C,
\]
for left-invariant vector fields \(A,B,C \in \mathfrak{X}_L(\SO(3))\). 

The rigidity of the Sasaki metric makes its geodesics ill-suited for path-planning purposes, since it is precisely the jet lifts that define the higher-order geometry of mechanical systems. Instead, it is often more useful to consider \emph{constrained} variational problems obtained by minimizing the energy functional of the Sasaki metric through jet lifts:
$$\min_q E^{(1)}[j^1 q] = \min_{q} \int_0^T \left(\|\dot{q}\|^2 + \|\nabla_{\dot{q}} \dot{q}\|^2 \right)dt,$$
the solutions to which can be interpreted as the closed-loop integral curves of a double-integrator system with an $L^2$-optimal control plus a regularization penalty:
$$\min_{u,q} \int_0^T \left( \|\dot{q}\|^2 + \|u\|^2\right)dt, \qquad \text{subject to } \nabla_{\dot{q}}\dot{q} = u.$$
To gain control over the regularization term $\|\dot{q}\|$, we consider the vector bundle isomorphism $F_\varepsilon: TM \to TM$ given by
$$F_\varepsilon(q,v) = \left(q,\,\frac1{\varepsilon}v \right), \qquad \varepsilon>0.$$
Pulling the Sasaki metric back through $F_\varepsilon$ provides a \emph{weighted} Sasaki metric on $TM$:
\begin{equation}\label{eq: weighted_Sasaki_metric}
    \llangle X,Y\rrangle
=
\langle \tau_{TM*}X,\tau_{TM*}Y\rangle
+
\frac1{\varepsilon^2}\langle K(X),K(Y)\rangle,
\end{equation}
whose jet-constrained geodesics solve:
\begin{equation}\label{eq: cubic_in_tension_variational_principle}
    \min_q \int_0^T \left(\varepsilon^2 \|\dot{q}\|^2 + \|\nabla_{\dot{q}}\dot{q}\|^2 \right)dt
\end{equation}
The solutions to which are known as \emph{Riemannian cubic polynomials in tension}, where $\varepsilon>0$ is the tension parameter \cite{SLCamCr95Proc, SLCamCr2000, NoakesPopiel2006}. The critical points satisfy the fourth-order ODE 
\begin{equation}\label{eq: cubic_tension}
    \nabla_{\dot{q}}^3 \dot{q} + \mathcal{R}(\nabla_{\dot{q}} \dot{q}, \dot{q})\dot{q} = \varepsilon^2\nabla_{\dot{q}} \dot{q}
\end{equation}

With respect to the standard bi-invariant metric on $\SO(3)$, equation \eqref{eq: cubic_tension} becomes
\begin{align*}
    \dot{R} &= R\hat{\Omega} \\
    \dddot{\Omega} + \Omega \times \Omega + \frac12 &\Omega \times(\Omega \times \dot{\Omega}) = \varepsilon\dot{\Omega}
\end{align*}


\subsection{Prolongation to the Second-Order Tangent Bundle}

Let \(\tau_{TM}:TM\to M\) and \(\tau_{TTM}:TTM\to TM\) denote the tangent bundle projections. The \emph{second-order tangent bundle} is the submanifold
\[
T^{(2)}M
=
\{u\in TTM \;:\; \tau_{TM*}(u)=\tau_{TTM}(u)\}.
\]
Equivalently, \(u\in T^{(2)}M\) may be identified with the \(2\)-jet of a curve \(c\) on \(M\), and in local coordinates \((q^i,v^i)\) on \(TM\) has the form
\[
u = v^i\frac{\partial}{\partial q^i}\Big|_{(x,v)}
    +u^i\frac{\partial}{\partial v^i}\Big|_{(x,v)} .
\]

There are two natural projections on \(T^{(2)}M\): the affine bundle
\[
\tau^2:=\tau_{TTM}|_{T^{(2)}M}:T^{(2)}M\to TM,
\]
and the fiber bundle
\[
\tau^1:=\tau_{TM}\circ\tau^2:T^{(2)}M\to M.
\]
These induce canonical vertical bundles
\[
V_1:=\ker(\tau^1_{*}),\qquad V_2:=\ker(\tau^2_{\ast}).
\]
The canonical almost tangent structure \(J:TT^{(2)}M\to TT^{(2)}M\) satisfies
\[
\operatorname{Im}(J)=V_1,\qquad J^2(X)=(\tau^1_{*}X)^{v_2},
\]
where \((\cdot)^{v_2}\) denotes the vertical lift to $V_2$, which is canonically defined since $\tau^2$ is an affine bundle. We refer to \cite{Suri} for background on connection maps in higher-order tangent geometry.

A connection map on \(T^{(2)}M\) is a pair \(\tilde K=(K_1,K_2)\) of fiberwise linear maps \(K_i:TT^{(2)}M\to TM\) satisfying
\[
K_2\circ J=K_1,\qquad K_1\circ J=\tau^1_{*}.
\]
The corresponding horizontal bundle is
\[
H:=\ker(\tilde K)=\ker K_1\cap \ker K_2,
\]
and one obtains the decomposition
\[
TT^{(2)}M = H \oplus J(H)\oplus V_2.
\]
Accordingly, the map \(\Phi:TT^{(2)}M\to TM\oplus TM\oplus TM\),
\[
\Phi(X):=(\tau^1_*X,K_1(X),K_2(X)),
\]
is a bundle isomorphism.

The Levi--Civita connection \(\nabla\) also induces a family of smooth bundle isomorphisms \(F_{\varepsilon_1,\, \varepsilon_2}:T^{(2)}M\to TM\oplus TM\),
\[
F_{\varepsilon_1,\, \varepsilon_2}([c])=
\left(c(0),\frac1{\varepsilon_1}\dot c(0),\frac1{2\varepsilon_2}\nabla_{\dot c}\dot c(0)\right),
\]
where $\varepsilon_1, \, \varepsilon_2 > 0$ are weights. Thus a curve \(\Gamma(t)\) on \(T^{(2)}M\) may be identified with a triple
\[
F_{\varepsilon_1,\, \varepsilon_2}(\Gamma(t))=(q(t), v(t),y(t)),
\]
where \(q\) is a curve on \(M\) and \(v,y\) are vector fields along \(q\).

We now define the second-order connection map induced by \(\nabla\). Let \(K:TTM\to TM\) be the connection map \eqref{eq: connection_map_TM} on \(TM\), and set
\[
K_1:=K\circ \tau^2_*.
\]
Further, for \(X\in T_uT^{(2)}M\), let \(\gamma(s,t)\) be a family of curves adapted to \(X\). We define
\[
K_2(X):= \frac12\nabla^2_{\partial_t\gamma}\partial_s\gamma\big|_{(0,0)}.
\]
\begin{theorem}
    The map $\tilde{K} = (K_1, K_2)$ is a connection map on $T^{(2)}M$. 
\end{theorem}

\begin{proof}
    It is easily seen that both $K_1$ and $K_2$ are fiberwise linear maps, and
$$K_1 \circ J = K \circ \tau_\ast^2 \circ J = K \circ J \circ \tau^2_\ast = \tau_{TM\ast} \circ \tau_\ast^2 = \tau^1_\ast$$
It remains to show that $K_2 \circ J = K_1.$
Observe that if $\gamma(s,t)$ is adapted to $X \in T_{u}T^{(2)}M$, then the curve $\beta(s,t) = \gamma(st,t)$ is adapted to $J(X)$. Indeed, $\beta(0,t) = \gamma(0,t)$, so that $u = [\beta(0,\cdot)]$, and
$$\partial_s \beta(0,0) = 0, \qquad \partial_s\partial_t\beta^i(0,0) = \partial_s \gamma^i(0,0)$$
which are precisely the coordinates of $J(X)$. Hence,
\begin{align*}
K_2(J(X)) &= \frac12\nabla^2_{\partial_t\beta}\partial_s\beta\big|_{(0,0)} = \frac12 \nabla^2_{\partial_t \beta} t \partial_s \gamma \big|_{(0,0)} \\
&= \frac12 \left( 2\nabla_{\partial_t \beta} \partial_s \gamma + t\nabla^2_{\partial_t \beta} \partial_s \gamma\right)\big|_{(0,0)} \\
&= \nabla_{\partial_t \gamma} \partial_s \gamma \big|_{(0,0)} = K_1(X)
\end{align*}
\end{proof}

If \(\Gamma(t)\) is identified with \((q(t),v(t),y(t))\), then
\[
K_1(\dot\Gamma)=\varepsilon_1\nabla_{\dot q}v,
\]
and, using the curvature identity,
\[
K_2(\dot\Gamma)
=
\varepsilon_2\nabla_{\dot q}y+\frac{\varepsilon_1}2\,\mathcal R(v,\dot q)v.
\]
Hence
\[
\Phi(\dot\Gamma)
=
\left(
\dot q,\;
\varepsilon_1\nabla_{\dot q}v,\;
\varepsilon_2\nabla_{\dot q}y+\frac{\varepsilon_1}2\,\mathcal R(v,\dot q)v
\right).
\]

This motivates the following prolongation of the Sasaki construction: we define the generalized weighted Sasaki metrics on \(T^{(2)}M\) by requiring the splitting
\[
TT^{(2)}M = H \oplus J(H)\oplus V_2
\]
to be orthogonal, with each summand identified with \(TM\) through \(\Phi\). Equivalently,
{\small\begin{equation}\label{eq: Weighted_Sasaki_T2}
    \llangle X,Y\rrangle^{(2)}
=
\langle \tau^1_*X,\tau^1_*Y\rangle
+
\frac1{\varepsilon_1^2}\langle K_1(X),K_1(Y)\rangle
+
\frac1{\varepsilon_2^2}\langle K_2(X),K_2(Y)\rangle .
\end{equation}}
Since all such metrics are isometric and induce the same geodesics (up to isometry), we study the case $\varepsilon_1=\varepsilon_2 =1$ for simplicity. If \(F_{1,1}(\Gamma)=(q,v,y)\), the corresponding energy is
\begin{align*}
    E^{(2)}[\Gamma]
&=
\frac12\int_0^T \llangle \dot\Gamma,\dot\Gamma\rrangle^{(2)}\,dt \\
&=
\frac12\int_0^T
\left(
\|\dot q\|^2
+\|\nabla_{\dot q}v\|^2
+\left\|\nabla_{\dot q}y+\frac12\,\mathcal R(v,\dot q)v\right\|^2
\right)\,dt .
\end{align*}
Letting
$$\mathcal{J}[q,v,y] := \frac12 \int_0^T \|\nabla_{\dot{q}}y + \frac12\mathcal{R}(v,\dot{q})v\|^2dt$$
we see that $\delta E^{(2)}[q,v,y] = \delta E^{(1)}[q,v] + \delta J[q,v,y]$. We previously showed that
$$\delta E^{(1)}[q,v] = -\int_0^T \left(\langle \nabla_{\dot{q}} \dot{q} + \mathcal{R}(v,\nabla_{\dot{q}} v)\dot{q}, \delta q\rangle + \langle \nabla_{\dot{q}}^2 v, \delta v\rangle \right)dt$$
Similarly, we calculate:
\begin{align*}
    \delta \mathcal{J}[q,v,y] = & \int_0^T \Big{\langle} \nabla_{\partial_\varepsilon q_\varepsilon} \nabla_{\dot{q}} y_\varepsilon + \frac12 (\nabla_{\partial_\varepsilon q_\varepsilon} R)(v_\varepsilon,\dot{q}_\varepsilon)v_\varepsilon \\
    & + \frac12 \mathcal{R}(\nabla_{\partial_\varepsilon q_\varepsilon} v_\varepsilon, \dot{q}_\varepsilon)v_\varepsilon  + \frac12 \mathcal{R}(v_\varepsilon, \nabla_{\partial_\varepsilon q_\varepsilon}\dot{q}_\varepsilon)v_\varepsilon \\
    &+ \frac12 \mathcal{R}(v_\varepsilon,\dot{q}_\varepsilon)\nabla_{\partial_\varepsilon q_\varepsilon} v_\varepsilon, \nabla_{\dot{q}} y_\varepsilon
    + \frac12 \mathcal{R}(v_\varepsilon,\dot{q}_\varepsilon)v_\varepsilon\Big{\rangle} dt \\
    = &-\int_0^T \langle \nabla_{\dot{q}}^2 y + \frac12 (\nabla_{\dot{q}} R)(v,\dot{q})v  \\
    &+ \frac12 \mathcal{R}(\nabla_{\dot{q}}v, \dot{q})v + \frac12 \mathcal{R}(v, \nabla_{\dot{q}}\dot{q})v \\
    &+ \frac12 \mathcal{R}(v, \dot{q})\nabla_{\dot{q}}v, \delta y\rangle dt \\
    &  -\frac12 \int_0^T \left\langle R\left(v, Y\right)v + \mathcal{R}(v, \dot{q})Y, \delta v\right\rangle \\
    & - \frac12 \int_0^T \left\langle (\nabla_v R)(v, Y)\dot{q} +(\nabla_{\dot{q}}R)(Y,v)v \right.\\
    & \left. + \mathcal{R}(Y,v)v + \mathcal{R}(y,Y)\dot{q}, \delta q \right\rangle dt
\end{align*}
where $Y = \nabla_{\dot{q}}y + \frac12 \mathcal{R}(v, \dot{q})v$. Applying the fundamental lemma of the calculus of variations, we obtain the following result.

%
%

\begin{theorem}
    $\Gamma = F_{1,1}^{-1}\circ (q,v,y)$ is a geodesic of the Sasaki metric if and only if
    \begin{align*}
        &2\nabla_{\dot{q}} \dot{q} + 2\mathcal{R}(v, V)\dot{q} + 2\mathcal{R}(y,Y)\dot{q} + (\nabla_{v}R)(v,Y)v \\
        &+ \mathcal{R}(V, Y)v  + \mathcal{R}(v,Y) V= 0, \\
       & 2\nabla_{\dot{q}} V +\mathcal{R}(v,Y)v + \mathcal{R}(v,\dot{q})Y = 0, \\
        &\nabla_{\dot{q}}Y = 0,
    \end{align*}
    where  $V= K_1(\dot{\Gamma}) =  \nabla_{\dot{q}} v$ and $Y = K_2(\dot{\Gamma}) =  \nabla_{\dot{q}} y + \frac12 \mathcal{R}(v, \dot{q})v$.
\end{theorem}

\begin{remark} Let $\Gamma$ be a geodesic on $T^{(2)}M$.
\begin{enumerate}
\item If  $\Gamma$ is a  $2$-horizontal geodesic on $T^{(2)}M$ (that is, $Y = K_2(\dot{\Gamma}) = 0$), then the projection of $\Gamma$ on $TM$  is a geodesic.
\item If  $\Gamma$ is a  $1$-horizontal geodesic on $T^{(2)}M$ (that is, $V= K_1(\dot{\Gamma}) =0$ and $Y = K_2(\dot{\Gamma}) = 0$), then the projection of $\Gamma$ on $M$  is a geodesic.
\end{enumerate}
\end{remark}

On the other hand, if $\Gamma$ is the $2$nd-jet lift of a curve $q$ on $M$, then the vector fields $v$ and $y$ along $q$ are
$$v=\dot{q}, \quad y=\frac 1 2 \nabla_{\dot{q}}\dot{q} $$ and the velocity vector field of $\Gamma$ satisfies
  $$
      V= K_1(\dot{\Gamma}) =  \nabla_{\dot{q}} \dot q, \quad  Y = K_2(\dot{\Gamma}) =  \frac 1 2\nabla_{\dot{q}}^2 \dot{q}.
$$

\begin{corollary}
    Let $q$ be a curve on $M$. The $2$nd-jet lift $\Gamma$ of  $q$   is a geodesic on $T^{(2)}M$ iff  $q$ is a geodesic on $M$.
\end{corollary}

\begin{proof}
 Let $\Gamma$  be the $2$nd-jet lift $j^2q$ of a curve  $q$ on $M$. We suppose that $\Gamma$ is a geodesic on $T^{(2)}M$. Then, $\Gamma$ satisfies the geodesic equations
\begin{align*}
      0=& \nabla_{\dot{q}}\dot{q} + \mathcal{R}(\dot{q}, \nabla_{\dot{q}}\dot{q})\dot{q} + \frac12 \mathcal{R}(\nabla_{\dot{q}}\dot{q}, \nabla_{\dot{q}}^2\dot{q})\dot{q} \\
      &
      + \frac14 \mathcal{R}(\dot{q}, \nabla_{\dot{q}}^2\dot{q})\nabla_{\dot{q}}\dot{q}
     +\frac 1 2(\nabla_{\dot{q}} R)(\dot{q}, \nabla_{\dot{q}}^2\dot{q})\dot{q},\\
       0=&\nabla_{\dot{q}}^2\dot{q} + \frac14 \mathcal{R}(\dot{q},\nabla_{\dot{q}}^2\dot{q})\dot{q}, \\
       0=& \nabla_{\dot{q}}^3\dot{q}.
    \end{align*}
Taking the inner product with $\dot q$ in the first two equations,
\begin{align}
  \label{eqf1aa}
&\langle\nabla_{\dot{q}}\dot{q}, \dot{q}\rangle=-\langle\nabla_{\dot{q}}^2\dot{q}, \nabla_{\dot{q}}\dot{q}\rangle, \\
&\label{eqf1bb}\langle\nabla_{\dot{q}}^2\dot{q}, \dot{q}\rangle=0.
\end{align}
Differentiating Equation (\ref{eqf1bb}), we get
 $\langle\nabla_{\dot{q}}^2\dot{q}, \nabla_{\dot{q}}\dot{q}\rangle=0$, which combined with (\ref{eqf1aa}), gives $\langle\nabla_{\dot{q}}\dot{q}, \dot{q}\rangle=0$. Finally, deriving the last equation and using Equation \eqref{eqf1bb}, we  conclude that
$\nabla_{\dot{q}}\dot{q}=0.$
\end{proof}

Returning to the example of $M =\SO(3)$ with its standard bi-invariant metric, the second-order tangent bundle $T^{(2)}\SO(3)$ can be trivialized via left-translations,  $T^{(2)}\SO(3)\cong \SO(3) \times \R^6$: A curve $\Gamma$ in $T^{(2)}\SO(3)$ can be  identified with a triple of curves $(R, \eta, \zeta)$, where  $R$ is a curve in $\SO(3)$ and $\eta$ and $\zeta$  are curves in $\R^3$ given by $F_{1,1}\circ \Gamma=(R, R\hat{\eta}, \frac 12 R\hat{\zeta})$. The velocity vector field can  also be identified with the triple $(\Omega, \Lambda, \Upsilon)$ of curves in $\R^3$ where $\dot{R} = R\hat{\Omega}$, $K_1\circ \Gamma= R\hat{\Lambda}$ and $K_2\circ \Gamma= R\hat{\Upsilon}$. Geodesics on  $\SO(3) \times \R^6$ are the curves $(R, \eta, \zeta)$ that satisfy
\begin{align*}
\dot{R} & =  R\hat{\Omega}, \; \;  \dot \eta =\Lambda+\frac 1 2 \eta \times \Omega,  \; \;
\dot \zeta = \Upsilon+\frac 1 2 \zeta \times \Omega+\frac 1 8 (\eta \times \Omega) \times \eta\\
\dot \Omega & = \frac 1 4 (\eta\times\Lambda)\times \Omega +\frac 1 4 (\zeta\times\Upsilon)\times \Omega))+\frac 1 8 (\Lambda\times\Upsilon)\times\Omega\\
& \quad+\frac 1 8 (\eta\times\Upsilon)\times\Lambda\\
\dot \Lambda & = \frac 1 2 \Lambda\times\Omega +\frac 1 8 (\eta\times\Upsilon)\times \eta+\frac 1 8 (\eta\times\Omega)\times\Upsilon, \; \; \dot \Upsilon  = \frac 1 2 \Upsilon\times\Omega.
\end{align*}

As on $TM$, we can consider the constrained variational problem obtained by constraining the generalized Sasaki energy through $2$-jets of curves. Since $2$-jets are not preserved under the mappings $F_{\varepsilon_1, \varepsilon_2}$, we return to the energy function of the \emph{weighted} generalized Sasaki metric:
$$\min_q E^{(2)}_{\varepsilon_1,\, \varepsilon_2}[j^2 q] = \min_q \int_0^T \! \big( \|\dot{q}\|^2 + \frac1{\varepsilon_1^2}\|\nabla_{\dot{q}} \dot{q}\|^2 + \frac1{4\varepsilon_2^2} \|\nabla_{\dot{q}}^2 \dot{q}\|^2\big)\, dt$$
It is clear that, up to renaming and scalar multiplication, the minimizers agree with those of the following variational problem:
\begin{equation}\label{eq: quintic_in_tension_variational_principle}
    \min_q \int_0^T \big(\varepsilon_1^2 \|\dot{q}\|^2 + \varepsilon_2^2 \|\nabla_{\dot{q}}\dot{q}\|^2 + \|\nabla^2_{\dot{q}}\dot{q}\|^2 \big)\, dt
\end{equation}
which can be understood as a variational principle defining Riemannian \emph{quintics} in tension, now with the two tension parameters $\varepsilon_1, \varepsilon_2 > 0$ which determine the weighting of the regularization terms in the corresponding triple-integrator optimal control problem:
\begin{equation}\label{eq: optimal_control_quintics_tension}
    \min_{q,u} \int_0^T \left(\varepsilon_1^2 \|\dot{q}\|^2 + \varepsilon_2^2 \|\nabla_{\dot{q}}\dot{q}\|^2 + \|u\|^2 \right)dt, \quad \text{subject to }\, \nabla^2_{\dot{q}}\dot{q}=u .
\end{equation}
\begin{theorem}
    $q = \displaystyle \arg\min_{c} E^{(2)}_{\varepsilon_1,\, \varepsilon_2}[j^2 c]$ over the space of $C^3$, piecewise smooth curves connecting fixed boundary conditions in position, velocity, and covariant acceleration at $t = 0$ and $t = T$ only if
    \begin{equation}\label{eq: quintic_tension}
\begin{split}
    \nabla_{\dot{q}}^5 \dot{q} + \mathcal{R}(\nabla_{\dot{q}}\dot{q}, \nabla^2_{\dot{q}}\dot{q})\dot{q} + \mathcal{R}&(\nabla_{\dot{q}}^3\dot{q}, \dot{q})\dot{q} = \\
    &\varepsilon_2^2\nabla^3_{\dot{q}}\dot{q} + \varepsilon_2^2\mathcal{R}(\nabla_{\dot{q}}\dot{q}, \dot{q})\dot{q} - \varepsilon_1^2 \nabla_{\dot{q}}\dot{q}
\end{split}
\end{equation}
\end{theorem}
\begin{proof}
    From the previous analysis on cubics in tension, it follows easily that
    \begin{align*}
        \delta \int_0^T &\left(\varepsilon_1^2\|\dot{q}\|^2 + \varepsilon_2^2\|\nabla_{\dot{q}}\dot{q}\|^2 \right)dt \\ 
        &= -\int_0^T \langle \varepsilon_2^2\nabla_{\dot{q}}^3 \dot{q} + \varepsilon_2^2\mathcal{R}(\nabla_{\dot{q}} \dot{q}, \dot{q})\dot{q} -\varepsilon_1^2\nabla_{\dot{q}} \dot{q},\delta q\rangle dt
    \end{align*}
    Moreover, in \cite{SLCamCr95}, it is shown that Riemannian quintic polynomials satisfy:
    {\small$$\delta \int_0^T \|\nabla_{\dot{q}}^2 \dot{q}\|^2dt = \int_0^T \langle\nabla_{\dot{q}}^5 \dot{q} + \mathcal{R}(\nabla_{\dot{q}}\dot{q}, \nabla^2_{\dot{q}}\dot{q})\dot{q} + \mathcal{R}(\nabla_{\dot{q}}^3\dot{q}, \dot{q})\dot{q}, \ \delta q\rangle dt$$}
    Combining the two yields the result.
\end{proof}

On $\SO(3)$ with its standard bi-invariant metric, equation \eqref{eq: quintic_tension} can be represented as:
\begin{equation}\label{eq: quintics_tension_biinvariant}
    \begin{split}
       \dot{R} = R\hat{\Omega}, \qquad \ddot{\Omega} &= \eta - \frac12 \Omega \times \dot{\Omega}, \qquad \dot{\eta} = \zeta - \frac12 \Omega \times \eta\\
    \ddot{\zeta} + \Omega \times \dot{\zeta} + \frac12 \dot{\Omega}& \times \zeta + \frac12 \Omega \times (\Omega \times \zeta) + \frac14 (\dot{\Omega} \times \eta) \times \Omega = \\
    &\qquad\qquad \varepsilon_2^2 \zeta + \varepsilon_2^2\Omega \times (\Omega \times \dot{\Omega}) - \varepsilon_1^2\dot{\Omega} 
    \end{split}
\end{equation}
Taking $\varepsilon_1 = \varepsilon_2 = 0$, we recover the left-trivialized equations of Riemannian quintic polynomials on $\SO(3)$ \cite{Zefran98, Krakowski2002}.



\section{Rigid Body Actuation}

Rigid body motion in \(\R^3\) is often modeled on \(\SE(3) \cong \SO(3) \times \R^3\), with $\SO(3)$ representing the rotational dynamics and $\R^3$ representing the translational dynamics. In many robotic applications, such as quadrotor control, the attitude dynamics can be viewed as a fully-actuated subsystem whose state affects the direction of applied thrust in the (underactuated) translational dynamics:
\begin{align}
    \ddot{x} + \nabla V(x) &= fRe_3, \label{eq: translation_dyn} \\
    \dot{R} &= R\hat{\Omega}, \label{eq: rotational_kin} \\
    \J \dot{\Omega} &= \J\Omega \times \Omega + \tau, \label{eq: rotational_dyn}
\end{align}
where \((R,x)\in \SO(3) \times \R^3\), \(V\in C^{1,1}(\R)\) is the potential energy, \(f\in C^{0,1}([0,T],\R)\) is the thrust input, \(\tau\in C^{0,1}([0,T],\R^3)\) is the torque input, and \(\Omega\in\R^3\) is the body angular velocity. It is therefore important to study trajectory-tracking problems for the rotational subsystem \eqref{eq: rotational_kin}-\eqref{eq: rotational_dyn} on $\SO(3)$ before advancing to the full model. In \cite{Lee}, a geometric torque control $\tau_d$ was proposed which yields near-global exponential tracking of a desired attitude \(R_d\).

This model is, however, an idealization: the inputs \(f\) and \(\tau\) are treated as directly applied forces and torques, whereas in practice they are generated through actuator dynamics. The commanded input acts on internal actuator variables, such as voltage or current, and the resulting torque is produced only indirectly. This introduces additional dynamics, as well as magnitude and rate constraints on the applied torque. A simple first-order model for the torque actuator is
\[
C\dot{\tau}+\tau=KM,
\]
where \(C,K\in\R\) and \(M\in C^{0,1}([0,T],\R^3)\) is the control input acting on the actuator \cite{CubeSats, spacecraft}. Hence, the effective rotational dynamics become the third-order system:
\begin{equation}\label{eq: effective_control_dynamics}
\begin{split}
    \dot{R} &= R\hat{\Omega}, \qquad \J\dot{\Omega} = \J\Omega \times \Omega + \tau, \\
    &\qquad\dot{\tau} = \frac{1}{C}(KM - \tau),
\end{split}
\end{equation}
For a general actuator input \(M\), we define the \emph{torque error}
\[
e_\tau:=\tau-\tau_d,
\]
which acts as an additive disturbance in the closed-loop attitude subsystem:
\[
\dot R = R\hat\Omega,\qquad
J\dot\Omega = J\Omega\times\Omega + \tau_d + e_\tau.
\]

Consider the actuator feedback
\begin{equation}\label{eq: control_design_M}
    M=\frac1K\big(\tau + C\dot\tau_d - Ck_\tau e_\tau\big),
\end{equation}
which yields the closed-loop error dynamics
\[
\dot e_\tau = -k_\tau e_\tau,
\]
and therefore
\[
\tau(t)=\tau_d(t)+e_\tau(t),\qquad
\|e_\tau(t)\|\le \|e_\tau(0)\|e^{-k_\tau t}.
\]
Hence the actuator error decays exponentially to zero. Since the nominal attitude tracking dynamics are exponentially stable, standard cascade arguments imply that the full actuator-attitude system retains near-global exponential tracking; that is, the attitude tracking errors converge exponentially to zero.

On the other hand, \eqref{eq: rotational_kin}-\eqref{eq: rotational_dyn} can be interpreted as the left-trivialization of the double-integrator system
$$\nabla_{\dot{R}} \dot{R} = R\tilde{\tau}$$
where $\nabla$ is the Levi-Civita connection of the left-invariant rigid body metric on $\SO(3)$ with inertia tensor $\J$, and $\tilde{\tau} := \widehat{\J^{-1}\tau}$. Taking another covariant derivative thus yields:
\begin{align*}
    \nabla_{\dot{R}}^2 \dot{R} &= R\left(\J^{-1} \dot{\tau} + \nabla^{\so(3)}_{\Omega} \J^{-1}\tau)\right)^\wedge \\
    &= R\left( \frac{K}{C}\J^{-1}M - \frac1{C}\J^{-1}\tau + \nabla_{\Omega}^{\so(3)}\J^{-1}\tau\right)^\wedge
\end{align*}
where $\nabla^{\so(3)}_{\Omega} \tilde\tau := \frac12\Omega \times \tilde\tau + \frac12 \J^{-1}(\Omega \times \J \tilde\tau + \tilde\tau \times \J\Omega)$. Consider the change of variables via feedback:
\begin{equation}\label{eq: M_transform_u}
    M = \frac{\tau}{K} + \frac{C}{K}\J\left(u - \nabla^{\so(3)}_{\Omega}\J^{-1}\tau\right)
\end{equation}
where $u \in C^{0,1}([0,T], \R^3)$ is an artificial control input to be designed.
We thus obtain the triple-integrator system
$$\nabla^2_{\dot{R}} \dot{R} = R\hat{u}.$$
Hence, the effective control system is of the form $K_2(j^2 R) = \frac12R\hat{u}$. 
Under the control design \eqref{eq: control_design_M}, it follows that
$$u = \J^{-1}(\dot{\tau}_d - k_\tau e_\tau) + \nabla^{\so(3)}_{\Omega} \J^{-1}\tau$$
results in the attitude tracking errors converging exponentially to zero. Given such a trajectory tracking control design, a common goal is to subsequently design the desired attitude trajectory $R_d$ such that energy consumption is minimized, often with regularity penalties on the derivatives of $R_d$ to ensure sufficient smoothness. As a proxy for energy minimization, we consider the optimal control problem
\begin{equation}\label{eq: optimal_control_final}
    \min_{u,R}\frac12\int_0^T
\left(
\varepsilon_1^2\|\dot R\|^2
+\varepsilon_2^2\|\nabla_{\dot R}\dot R\|^2
+\|u\|^2
\right)\,dt
\end{equation}

subject to \(\nabla_{\dot R}^2\dot R=R\hat u\) and boundary conditions in position, velocity, and covariant acceleration. The corresponding critical points are Riemannian quintics in tension with parameters \(\varepsilon_1,\varepsilon_2\). These are not given by the bi-invariant equations \eqref{eq: quintics_tension_biinvariant}, since the rigid body metric is left-invariant with inertia tensor \(\J\). In body coordinates, the induced connection on \(\so(3)\cong\mathbb R^3\) is given by \cite{goodman2022reduction}
\[
\nabla^{\so(3)}_\xi\eta
=
\frac12\left(
\xi\times\eta
-
\J^{-1}(\J\xi\times\eta + \J\eta\times\xi)
\right).
\]
Similarly, the curvature tensor is left-equivariant, and is its restriction to the Lie algebra satisfies:
\[
\mathcal R(\xi,\eta)\sigma
=
\nabla^{\so(3)}_\xi\nabla^{\so(3)}_\eta\sigma
-
\nabla^{\so(3)}_\eta\nabla^{\so(3)}_\xi\sigma
-
\nabla^{\so(3)}_{\xi\times\eta}\sigma.
\]
Now define
\[
\Omega^{(\alpha)}:=\big(R^T\nabla_{\dot R}^{\alpha}\dot R\big)^\vee,\qquad \alpha=0,\dots,4.
\]
Then
\[
\Omega^{(\alpha)}
=
\dot\Omega^{(\alpha-1)}
+
\nabla^{\so(3)}_{\Omega^{(0)}}\Omega^{(\alpha-1)},
\]
and the quintic-in-tension equation is equivalent to
\begin{equation}\label{eq: quintics_in_tension_left-invariant}
\begin{split}
\dot R &= R\hat\Omega^{(0)},\\
\dot\Omega^{(\alpha)}
&=
\Omega^{(\alpha+1)}
-
\nabla^{\so(3)}_{\Omega^{(0)}}\Omega^{(\alpha)},
\qquad \alpha=0,\dots,3,\\
\dot\Omega^{(4)}
&+
\nabla^{\so(3)}_{\Omega^{(0)}}\Omega^{(4)}
+
\mathcal R(\Omega^{(1)},\Omega^{(2)})\Omega^{(0)}
+
\mathcal R(\Omega^{(3)},\Omega^{(0)})\Omega^{(0)} \\
&=
\varepsilon_2^2\Omega^{(3)}
+
\varepsilon_2^2\mathcal R(\Omega^{(1)},\Omega^{(0)})\Omega^{(0)}
-
\varepsilon_1^2\Omega^{(1)}.
\end{split}
\end{equation}

\subsection{Numerical Simulations}

We numerically compare two choices of nominal attitude trajectories on $\SO(3)$: Riemannian cubic polynomials, and Riemannian quintics in tension, both computed with respect to the left-invariant metric with inertia tensor
\[
    \J = \operatorname{diag}(0.082,\;0.0845,\;0.1377),
\]
which reflects the experimentally measured inertia tensor of a type of quadrotor UAV \cite{Lee}. Riemannian cubics were chosen as a baseline, as they represent a standard choice in nominal trajectory for rigid body control systems without actuator dynamics. As such, it is expected that the quintics in tension will outperform Riemannian cubics in contexts where the actuator dynamics become highly relevant to the tracking problem, which can correspond to large $C$, small $K$, and aggressive boundary conditions in position and velocity. 

We note that the equations \eqref{eq: M_transform_u} and \eqref{eq: rotational_dyn} inform natural choices in tension parameters for \eqref{eq: optimal_control_final}. Indeed, it is easily seen that
$$M = \frac1{K}\J\left(\Omega^{(1)} + Cu - C\nabla^{\so(3)}_{\Omega^{(0)}}\Omega^{(1)} \right)$$
so that the true control input $M$ is approximately minimized with respect to the optimal proxy control $u^\ast$ solving \eqref{eq: optimal_control_final} with $\varepsilon_1 \ll 1$ small and $\varepsilon_2 \propto 1/C$. We chose the particular values of $K = 0.15$, $C = 0.75$, $\varepsilon_1 = 0.01$, and $\varepsilon_2 = 0.5$ for our simulations.

Riemannian cubics in general can only be made to solve boundary conditions in position and velocity, whereas quintics in tension can solve additional boundary conditions in accelerations. To make for a fair comparison between the two, we choose ``natural" quintics in tension, which satisfy the given boundary conditions in position and velocity in addition to $\Omega^{(2)}(0) = \Omega^{(2)}(T) = 0$. Such conditions minimize \eqref{eq: optimal_control_final} with free boundary conditions in acceleration.

For both Riemannian cubics and quintics in tension, the resulting integral curves were used as the nominal trajectory $R_d(t)$ for the closed-loop effective dynamics \eqref{eq: effective_control_dynamics} with the control input \eqref{eq: control_design_M}, and with the desired torque input given in \cite{Lee}. The same boundary conditions, controller gains, actuator parameters, and initial tracking offset were used in both cases. To study the tracking performance, we use the following notions of tracking errors:
\begin{align*}
    e_R &= \frac12 (R_d^T R - R^T R_d)^\vee, \quad 
    e_\Omega = \Omega - R^T R_d \Omega_d, \quad e_\tau = \tau - \tau_d.
\end{align*}
The physical norm $\|\cdot\|_\J$ of the left-invariant metric is used to calculate the norms of all tracking errors and control inputs.
\begin{figure}[t]
    \centering
    \includegraphics[width=1\linewidth]{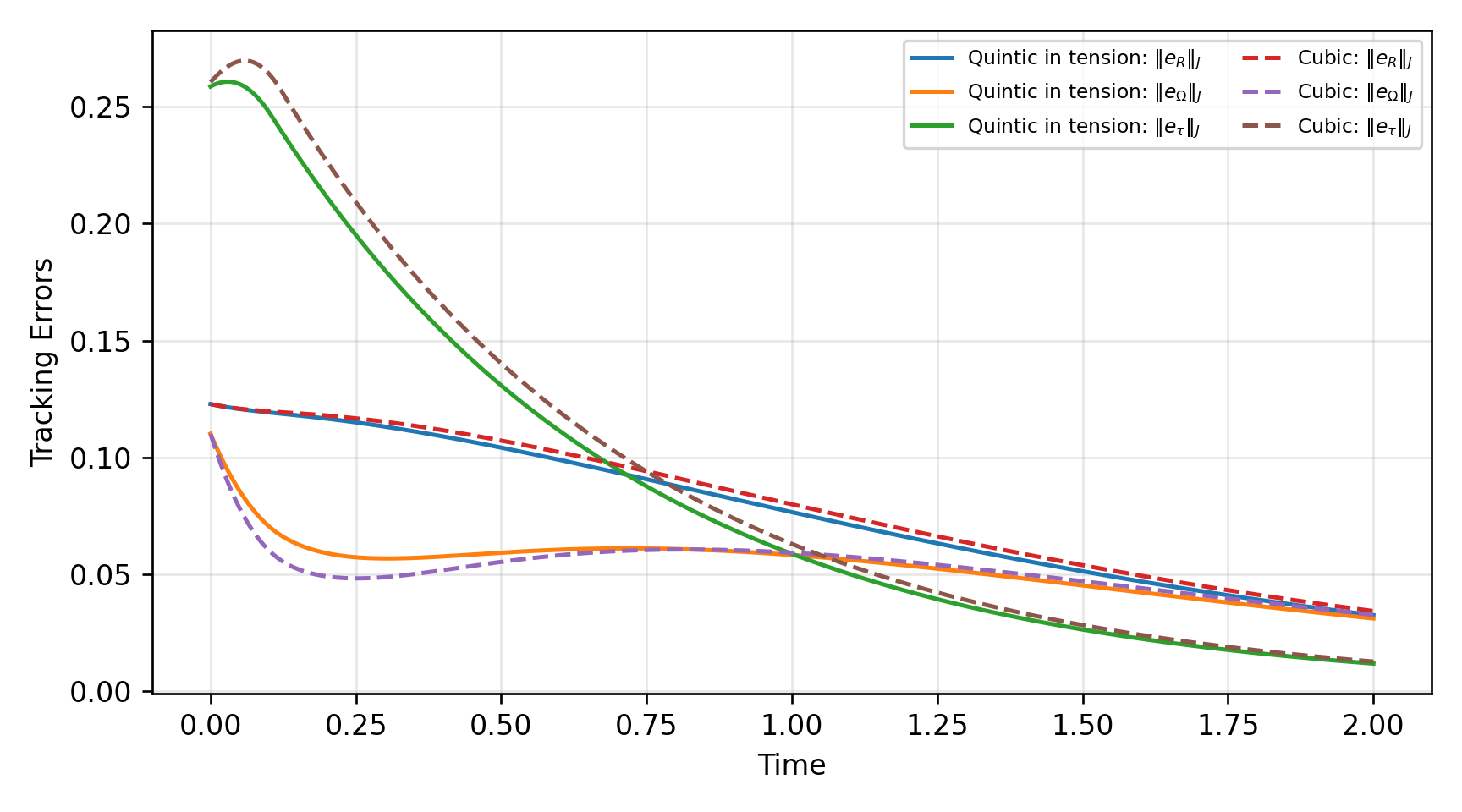}
    \includegraphics[width=1\linewidth]{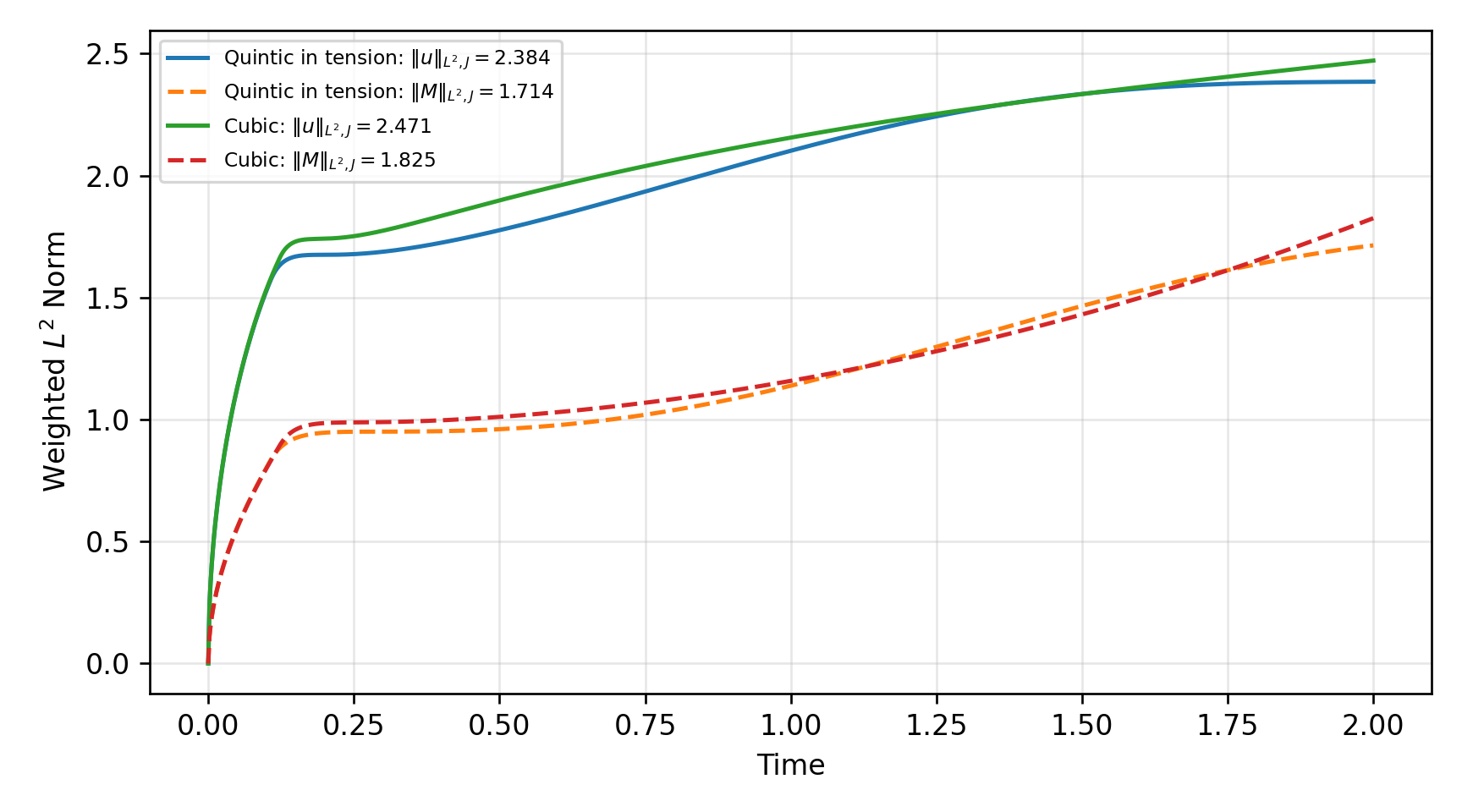}
    \caption{Comparison of Riemannian cubic and quintic-in-tension nominal
trajectories for attitude tracking with first-order actuator dynamics.
Top: tracking errors measured in the physical \(J\)-norm. Bottom:
cumulative weighted \(L^2\) norms of the higher-order proxy input and
the actuator command.}
    \label{fig:sims}
\end{figure}

As shown in Figure~\ref{fig:sims}, the closed-loop tracking behavior is
comparable for the two nominal trajectories\footnote{The code can be found at \url{https://doi.org/10.5281/zenodo.20730224}}. The Riemannian quintic in tension
produces a slightly smaller torque tracking error, while the Riemannian cubic
has a smaller initial angular velocity tracking error. The attitude tracking
error is roughly the same in both cases.
Thus, for this maneuver, the effect of the higher-order trajectory model is not
primarily visible in the attitude path or in the final tracking error.

The distinction is more apparent in the control diagnostics. In the second plot
of Figure~\ref{fig:sims}, we compare the cumulative weighted $L^2$ norms
\[
    t \mapsto \left(\int_0^t \|u(s)\|_{\J}^2\,ds \right)^{1/2},
    \qquad
    t \mapsto \left(\int_0^t \|M(s)\|_{\J}^2\,ds \right)^{1/2}.
\]
The quintic in tension yields smaller final values for both quantities,
indicating a modest reduction in the higher-order proxy input and in the actual
actuator command, which can correspond to measurable loss in energy consumption over many repeated trials.

We emphasize that this comparison should not be interpreted as solving the
exact minimum-command optimal control problem. Rather, Riemannian quintics in tension provide an
intrinsic higher-order baseline nominal trajectory motivated in the presence of first-order actuator dynamics. The
numerical results suggest that, under actuator-sensitive parameter choices and
aggressive boundary data, the quintic-in-tension trajectory can provide a small
but measurable improvement in actuator-relevant cost while preserving comparable
closed-loop tracking performance when compared to Riemannian cubic splines.

\end{document}